%% file: main-act.tex
\documentclass[ a4paper
              , 10pt
              , reqno
              , submission
			  ]{eptcs}
\providecommand{\event}{ACT2022}

\input{loca-macro-lics.tex}
\usepackage{doi}
\usepackage[ backend=bibtex
	         , natbib=true
	         , style=numeric
	         , sorting=nyt
	         , maxbibnames=10
	         , giveninits=true
	         , backref=true
	         , url=false
	         , doi=true
	         , isbn=false
	         , arxiv=abs
	         , date=year
	         , abbreviate = true
           ]{biblatex}

\AtBeginBibliography{\small}

\begin{document}
\maketitle
\begin{abstract}
We study the notion of a \emph{differential 2-rig}, a category $\clR$ with coproducts and a monoidal structure distributing over them, also equipped with an endofunctor $\d : \clR \to \clR$ that satisfies a categorified analogue of the Leibniz rule. This is intended as a tool to unify various applications of such categories to computer science, algebraic topology, and enumerative combinatorics. The theory of differential 2-rigs has a geometric flavour
but boils down to a specialization of the theory of tensorial strengths on endofunctors;
this builds a surprising connection between apparently disconnected fields. We build \emph{free 2-rigs} on a signature, and we prove various initiality results: for example, a certain category of colored species is the free differential 2-rig on a single generator.
\end{abstract}

\section{Introduction}
The aim of the present paper can be shortly summarized as follows: study a pair $(\clR,\d)$, where $\clR$ is a `categorified ring' and $\d : \clR \to \clR$ an endofunctor preserving coproducts and satisfying the `Leibniz rule'.

Adapting terminology from classical ring theory, such a pair $(\clR,\d)$ could be termed a \emph{differential 2-rig}, and $\d$ a \emph{derivation} on $\clR$; the study of such structures could thus be viewed as a categorified version of \emph{differential algebra} \cite[Ch. 1]{book:crespo}, an important part of modern commutative algebra \cite[III.10]{BourbakiAlgebre}, finding applications (among other areas) in Galois theory \cite{pommaret1994differential} and in symbolic computation \cite{chyzak:hal-01069833}.

Building on this, in our work, a 2-rig\footnote{An important difference with classical ring theory is that the request that $\clR$ admits `additive inverses' is an extremely restrictive one. This motivates our choice of terminology: a \emph{rig} $(R,+,\cdot)$ --also called a \emph{semiring}-- is a ri\emph{n}g without  \emph{n}egatives, i.e. an algebraic structure that satisfies all the axioms of a ring, but where $(R,+)$ is just a commutative monoid.} will be a category $\clR$ equipped with two structures, one additive and one multiplicative, such that the latter `distributes' over the former: at its most basic level, this is the requirement that, for objects $A,B\in\clR$, the endofunctors $A\otimes-$ and $-\otimes B$ distribute over coproducts, i.e. there are natural isomorphisms $A\otimes(B+C)\cong A\otimes B + A \otimes C$ and $(B+C)\otimes A \cong B \otimes A + C\otimes A$. Nevertheless, our main definition will be fairly more general, treating other shapes of colimits apart from this basic one.
\paragraph{Literature on 2-rigs.}
A motivating example of `categorified calculus on a 2-rig' is Joyal's theory of species and analytic functors \cite{Joyal1986foncteurs,joyal1981theorie,bergeron1998combinatorial} providing a categorical foundation for enumerative combinatorics and finding concrete applications as a model of PCF \cite{HASEGAWA2002113}.\footnote{The equivalence between analytic functors, regarded as `categorified formal power series', and species is a long-established result first proved in \cite{Joyal1986foncteurs}; see also \cite{adamek2008analytic}.}
The category of combinatorial species (functors $\Sigma^\op \to \Set$ from the category $\Sigma$ of finite sets and bijections) is a prominent example of a 2-rig which supports a viable notion of derivative functor, and it will always be our motivating example and test-bench for definitions.

This situates our work on a different ground than another important piece of literature dealing with notions of derivation on a category, namely the theory of \emph{differential categories} of Blute, Cockett et al. \cite{blute2009cartesian}. Differential categories were developed to provide a categorical doctrine for \emph{differential linear logic}; as a rule of thumb, the fundamental difference between the two approaches lies in where the categorified derivative operation acts. In differential categories, every \emph{morphism} has a derivative assigned via a so-called differential combinator; instead, we focus on deriving \emph{objects} functorially and naturally.\footnote{For the sake of completeness, we shall mention yet another approach to `categorical differentiation' recently developed in \cite{2103.07960} with applications to ZX calculus in mind; again, there seems to be no relation with our theory of differential 2-rigs, since derivations on their category Mat-$S$ are not Leibniz on objects.}

Elsewhere, terms like `2-rig' or `rig categories' have been appropriated by different authors to mean different things. For example, \cite{HDA3} defines a 2-rig to be a cocomplete symmetric monoidal category in which the monoidal product distributes over all colimits, and in \cite{BMT2021}, `2-rig' has meant a $\cate{Vect}$-enriched symmetric monoidal category with biproducts and idempotent splittings (where the distributivity is automatic). On the other hand, the term `rig category' or `distributive category' \cite{laplaza1972coherence} has been used to mean a category with two monoidal products, one called `multiplicative', which distributes over the other, called `additive'. It is easy to imagine variations on these themes: 2-rigs which are only fi\-ni\-te\-ly cocomplete, or that are assumed only to have \emph{finite} coproducts (which we consider to be a baseline assumption).

Alternatively, on the multiplicative side, one might want infinite products and a complete distributivity law over infinite coproducts. This type of `2-rig' would be germane to the study of polynomial functors in the sense of \cite{gambino2013polynomial,kock2017polynomial,myers2020dirichlet,spivak2020dirichlet,spivak2020poly}, which have provided a unifying setting for studying numerous structures in applied category theory.
Given the multiplicity of possible definitions of 2-rig, we believe it makes sense not to fix a single notion of 2-rig but to be flexible and contemplate a whole spectrum of possible theories, or `doctrines' of \emph{$\bsD$-rigs} parametrized by $\bsD$, a 2-monad on $\Cat$ (locally small categories) whose algebras will possess colimits of a certain shape.
\paragraph{Our main contributions.}
The first goal of this paper is to provide a generalized framework in which each of these instances can be studied on the same foot; our main definition for a `doctrine of 2-rigs', \ref{doctrine_of_2rigs}, is geared in this direction.
Besides unifying most notions of 2-rig under a common framework, in this paper we are also interested in seeing how different $\bsD$-rigs $\clR$, for different doctrines $\bsD$, interact with an accompanying notion of derivation $\d : \clR \to \clR$ (roughly, functors which obey a Leibniz rule, $\d(A\otimes B)\cong \d A\otimes B + A\otimes \d B$, see \ref{def_derivation}). Unexpectedly, depending on the doctrine, derivations may be either virtually non\-existent (cf. \ref{nogo1} and \ref{nogo2}), as is the case when the multiplicative structure is cartesian, or may exist in great plenitude, typically when the multiplicative structure enjoys a more `linear' character (in the sense of Girard's linear logic \cite{Gir1987}). Within a doctrine $\bsD$ where derivations are prevalent, they may also be used to give a notion of `dimension of a $\bsD$-rig' (cf. \ref{dim_of_rigs}).

In such cases, one generally expects derivations to be potent and unifying tools. We show that this is the case, once again guided by the theory of species as motivating example, and the usage of derivatives in the hands of the `Montreal school of categorical combinatorics' \cite{a0e13a522d655a338dcd3eb0d86f31d83d15c9bd,10.1007/BFb0072518,0445243bd75f64484c47f7db18f2569031b5e3bd,4871af04f741391c618dddc08247a3c09b1c707d,7c97349a4b1a346b1746f0519c000954171f7c20} (see also the more recent \cite{RAJAN1993197,a7260d9aa6cf11af5e8a6e16d6cb29f836d2ff04}), where differential equations written in the category of species, as well as their solutions, are fruitfully interpreted combinatorially. We prove that categories of species are `necessary objects' in a general theory of 2-rigs, because they arise as free objects for specific doctrines of 2-rigs and acquire a canonical choice of a differential structure. Moreover, in \ref{diff_poly_rig} we prove that the category of species on a countable set, equipped with a `shifting' derivation operation, is the free differential 2-rig on one generator.

Derivations can also be used to shed light on the theory of operads; for example, recent results by O\-bra\-do\-vich \cite{obrad2017} show that ordinary (permutative) operads are certain types of monoids for a skew monoidal structure $F' \otimes G$ defined using the derivative, and that cyclic operads \cite{GetzKap} also admit an efficient description in terms of derivatives of species.
In spite the large effort to understand the properties of a \emph{specific} instance of differential 2-rig (see \cite{a0e13a522d655a338dcd3eb0d86f31d83d15c9bd,10.1007/BFb0072518,0445243bd75f64484c47f7db18f2569031b5e3bd,4871af04f741391c618dddc08247a3c09b1c707d,7c97349a4b1a346b1746f0519c000954171f7c20} for the theory of ODEs in the category of species, and a variety of works by M. Fiore \cite{10.1007/978-3-540-31982-5_2,1506.06402,Fiore2012} that explored in detail the meaning of \emph{bijective proofs} in terms of datatype structures), a systematic study of general properties of differential 2-rigs (a `synthetic 2-rig theory', so to speak) has never been attempted.

Thus, one first aim of the present paper is to give all the various notions of 2-rig and derivations their proper due, while balancing generality and applicability, and unifying diverse approaches. In developing the rudiments of this framework, we aim to clarify what is specific to the category $\Spc$ of species and what instead follows from a general theory of 2-rigs and concentrate on the latter to generalize the former (to other doctrines, other flavours of monoidality, other flavours of species --colored \cite{mendez1993colored} or linear \cite{10.1007/BFb0072518}).

As a showcase example, in \ref{chain} we present a `chain rule' that categorifies the well-known calculus theorem $(f\circ g)'(x)=f'(g(x))g'(x)$ and that holds good across a broad spectrum of doctrines of 2-rigs, thus generalizing the chain rule true for species and proved by Joyal in his early works.

\paragraph*{Structure of the paper.}
In \ref{doc_of_tworigs} we introduce the main object of discussion of our paper: a 2-rig for a `doctrine' $\bsD$, i.e. a specified class of colimits, is a category having all colimits specified by $\bsD$, and a monoidal structure $\otimes$ that distributes over said colimits; we show how this formalism is capable of encompassing most of the various notions of 2-rig scattered in the literature; in \ref{modules_n_stre} we outline the fundamental definition in order to arrive at a definition of derivation on a $\bsD$-rig $\clR$, a pair of \emph{tensorial strengths} interacting well with each other; to the best of our knowledge, the characterization of tensorial strengths as lax natural transformations in \eqref{dis} and \eqref{dat} has not been accounted elsewhere. Section \ref{diff_2_basic} is the heart of the paper: a differential 2-rig is defined in \ref{def_derivation}; after this we concentrate on the major example of combinatorial species, in \ref{spe_exa},
and we prove that for every $\otimes$-monoid $M$, its derivative $\d M$ is a $M$-module. In \ref{free_2_rigs} we provide the construction of free $\bsD$-rigs and prove that free 2-rigs acquire differential structures (\ref{free_are_diff}) as well as various initiality results: for example, the category of ($S$-colored) species is the free cocomplete 2-rig on a single (on $|S|$) generator; in \ref{conclu} we draw the conclusions of the paper and sketch ideas for future development: the opportunity to gain a geometric view on applicatives, through derivations on a 2-rig seems to be a promising prospect, as well as the application of our general theory to a synthetic approach to combinatorial differential equations.
\section{Doctrines of 2-rigs}\label{doc_of_tworigs}

Before defining a notion of 2-rig doctrine, we present a few examples that play a guiding role and show a need for such a general notion.
\begin{example}[A list of motivating examples]\leavevmode\label{2rig}
	\begin{itemize}
		\item The category of presheaves $[\clM^\op, \Set]$ over a monoidal category $\clM$, equipped with the Day convolution, and its $\clV$-enriched analogue $[\clM^\op,\clV]$ \cite{day:report}. Note how the convolution product tends to inherit other structures of the monoidal product $\otimes$; e.g., the Day convolution $(F,G)\mapsto F * G$ is symmetric (or braided, or cartesian monoidal) if $\otimes$ is so, and it acts as the free monoidally cocomplete \cite{imkelly} category on $\clM$.
		\item The category of finite-dimensional vector bundles over a space or finitely generated projective modules over a commutative ring. These categories admit coproducts and tensor products, but not general colimits. Nor would one necessarily want to impose general colimits because of phenomena like `Eilenberg swindles' \cite{poenaru2007infinite}. These examples of `2-rigs' are typically enriched in vector spaces or the like, and typically the only colimits envisaged are \emph{absolute colimits}: those that are preserved by every (enriched) functor.
		\item Between these two extremes, one sometimes considers `2-rigs' which have colimits over diagrams that are bo\-und\-ed in size: for example, the categories of finite $G$-sets for some group $G$ that may be infinite, or of continuous finite $G$-sets for some topological group $G$, admit only finite colimits. Or, in the theory of locally $\kappa$-presentable categories, the subcategory of compact objects will admit colimits over diagrams bounded in size by $\kappa$.
	\end{itemize}
\end{example}
Guided by such examples, the following definitions are meant to encompass a spectrum of notions of 2-rigs that have arisen in practice.
\begin{definition}[Additive doctrine]\label{additive_doc}
	An \emph{additive doctrine} is a 2-category whose objects are categories that admit all colimits of diagrams belonging to a prescribed class, including at least finite discrete diagrams --whose colimits are finite coproducts, denoted with the infix $\cup$, and $0$ for the empty coproduct. Morphisms of a doctrine $\bsD$ are functors that preserve colimits of that class, and 2-cells are natural transformations between such functors.\footnote{A more general notion of additive doctrine is obtained by considering enriched analogues as well; in this paper, we mostly focus on the unenriched (i.e., $\Set$-enriched) case.}
\end{definition}
In each case, we may instead work with a stricter notion of additive doctrine where objects are categories with \emph{chosen} colimits: these are strict algebras of a strict 2-monad, which is often technically convenient. Strict algebra morphisms preserve those chosen colimits strictly, which is not what one wants, but pseudomorphisms preserve colimits in the usual sense \cite{Lack2010}.

So, an additive doctrine is determined by a (strict or pseu\-do) 2-monad $\bsA$ on $\Cat$, of which we consider the category of algebras. In short, the notion of an additive doctrine takes care of the additive monoid part of a 2-rig; as for the multiplicative part, we can similarly state the following definition.
\begin{definition}[Multiplicative doctrine]\label{mul_doc}
	A \emph{multiplicative doctrine} is a 2-category that is monadic (in the 2-categorical sense) over the 2-category $\cate{MCat}_s$ of monoidal categories, strong monoidal functors, and monoidal transformations, such that the composition of monadic functors,
	\[U_{\clM} = (\clM \to \cate{MCat}_s \to \cate{Cat}),\label{monadic_U}
	\]
	is also 2-monadic.
\end{definition}
Intuitively, a multiplicative doctrine consists of a category of monoidal categories, possibly equipped with additional structure, that arises as the category of algebras for a monad on $\Cat$. So, a multiplicative doctrine is given by a 2-monad $\bsM$ on $\Cat$ modelled over the 2-monad whose algebras are monoidal categories, of which we consider the 2-category of algebras.

The 2-category $\cate{MCat}$ of monoidal categories is trivially an example of multiplicative doctrine; so are the 2-category of symmetric, braided, or strict monoidal categories. For symmetric monoidal categories, algebra pseudomorphisms coincide with strong symmetric monoidal functors.\footnote{One might also want to replace $\cate{MCat}_s$ with the 2\hyp{}category $\cate{MCat}_l$ (having lax monoidal functors as 1-cells) or $\cate{MCat}_c$ (colax functors), but we do not explore such a generalization here. Also, it is well-known that the composition of monadic functors can fail to be monadic; to correct this shortcoming, various flatness conditions such as `preserving codescent objects' may be imposed on a (2-)monadic functor $G: \clM \to \cate{MCat}$ to guarantee that the composition $U_{\clM} = UG: \clM \to \cate{Cat}$ is also monadic, but this issue is somewhat technical, and it will not be pursued here.}
Finally, we need a notion of what it means for a multiplicative doctrine to \emph{distribute} over an additive doctrine. Intuitively, this is taken care by a \emph{distributive law} in the sense of \cite{beck1969distributive} between the two doctrines.

Let $\bsA$ be the 2-monad for any additive doctrine in the sense above, and let $\bfP$ be the 2-monad for the additive doctrine of all small-cocomplete categories, whose underlying functor $P$ takes a locally small category $\clC$ to the category consisting of presheaves $\clC^\op \to \Set$ that are small colimits of representable functors. We have an inclusion of 2-monads $j: \bsA \to \bfP$. Temporarily, let $\bsM$ denote the 2-monad whose algebras are monoidal categories, with underlying functor $M$.
Now, the Day convolution monoidal structure provides for each mo\-no\-i\-dal category $\clC$ a monoidal structure on the free small-cocompletion on its underlying category, $PU\clC$, and this construction also works as free cocompletion for monoidal categories \cite{imkelly}.
In other words, $PU\clC$ carries a canonical $\bsM$-algebra structure
$MPU\clC \to PU\clC$
pseudonatural in $\clC$, thus leading to an action
$MPU \To PU$
and such an action is equivalent to a canonical distributive law between monads
$\delta: \bsM\bfP \To \bfP\bsM$.
The only thing required to set up this distributive law is that Day-convolving on either side, $A * -$ or $- * A$, preserves all small colimits. This remains true \cite{szlachanyi2005monoidal} for any restricted class of colimits coming from an additive doctrine given by a monad $\bsA$ on $\Cat$; thus, we obtain by restriction a  distributive law
\[\delta': \bsM\bsA \To \bsA\bsM\]
or what is essentially the same, a canonical lifting $\hat{\bsA}$ of $\bsA$ as follows: there is a functor $\hat A : \cate{MCat} \to \cate{MCat}$ such that $U\circ \hat A \cong A \circ U$.
\begin{definition}
	A \emph{distributivity} of a multiplicative doctrine $\clM$ over an additive doctrine $\clA = \bsA\emdash\cate{Alg}$ is a choice of lift $\tilde{\bsA}$ in the diagram
	\[
		\vcenter{\xymatrix{
		\clM \ar[r]^{\tilde{A}} \ar[d] & \clM \ar[d] \\ \cate{MCat} \ar[r]_{\hat{A}} & \cate{MCat}
		}}
	\]
\end{definition}
Thanks to \cite[Definition 33, Remark 34]{walker2019distributive} such distributivities are essentially unique. This is particularly the case when the unit of the monad for the multiplicative doctrine over $\cate{MCat}$ is essentially surjective on objects (eso). In this case, the distributivity is uniquely given by the Day convolution structure at the underlying monoidal category level.

Now let $\bsM$ denote the monad for the monadic functor $U_{\clM} \to \cate{Cat}$ in \eqref{monadic_U}. As above, a distributivity $\tilde{A}$ amounts to an $\bsM$-action $MA U_{\clM} \To A U_{\clM}$, which corresponds to a 2-distributive law $\delta: \bsM\bsA \To \bsA\bsM$ between 2-monads.
\begin{definition}[Doctrine of 2-rigs]\label{doctrine_of_2rigs}
	A \emph{doctrine of 2-rigs} consists of an additive doctrine $\bsA$, a multiplicative doctrine $\bsM$, and a distributivity of $\bsM$ over $\bsA$.
\end{definition}
Using the distributive law, one obtains a structure of 2-monad $\bsA\bsM$ for the composition of functors. If one uses strict versions of the 2-monads $\bsA$ and $\bsM$, one obtains a strict 2-monad $\bsD = \bsA\bsM$, and a strict notion of $\bsD$-rig, with pseudomorphisms as an appropriate notion of \emph{morphism} of $\bsD$-rigs.
\begin{remark}
The original notion of `distributive category' given by Laplaza in \cite{laplaza1972coherence} is more general because it only asks for the presence of two monoidal structures (the `additive' one is not necessarily cocartesian). In such a setting, the complexity of the diagrams required to ensure coherence is daunting (cf. \cite[2.1.1]{elgueta2020groupoid}); our choice, where `coherence conditions' follow automatically from universal properties, avoids this problem by design.
\end{remark}
\begin{warning}
Whereas ordinary rigs form discrete distributive categories, ordinary rigs do not give discrete 2-rigs in our sense, since the only discrete category admitting finite coproducts is the singleton. Nor is a 2-rig with a single object an exciting notion: a monoidal category with a single object is a commutative monoid; by the Eckmann-Hilton argument, the two operations of 2-rig collapse in one single commutative monoid structure, for which multiplication $a\cdot\firstblank : M \to M$ is a monoid homomorphism.
\end{warning}
We can build on the definition of a doctrine of 2-rigs and turn our attention to some specific examples of interest, where we assume something more on the additive or multiplicative doctrine in study (symmetry, or the presence of more shapes of colimit).
\begin{definition}[Terminological conventions for doctrines of 2-rigs]\label{comm_coco_closed}\leavevmode
	\begin{itemize}
		\item A doctrine is \emph{symmetric} (resp., \emph{braided}, \emph{cartesian}) if the underlying multiplicative doctrine is the doctrine of symmetric (resp., braided, cartesian) monoidal categories.
		\item A doctrine is ($\kappa$-)\emph{cocomplete} if the class of all ($\kappa$-)small colimits gives the underlying additive doctrine.
		\item When the multiplicative doctrine is that of monoidal categories, then for an additive doctrine $\bsA$ we may refer to the 2-rigs as \emph{monoidally $\bsA$-cocomplete} categories.
		\item By default, `\emph{the}' doctrine of 2-rigs refers to the minimal notion of 2-rigs, where the multiplicative doctrine is just the doctrine of monoidal categories, and the additive doctrine is the $\omega$-additive doctrine.
		\item A \emph{closed} 2-rig is a category $\clR$ as in \ref{2rig} such that each $A\otimes -$ and $-\otimes B$ have right adjoints; in this case, of course, they preserve all colimits that exist in $\clR$.
	\end{itemize}
\end{definition}
\begin{notation}
With a small abuse of language, when we refer to a 2-rig as symmetric, cocomplete, \dots, we declare that we intend to consider it as an object of a 2-rig doctrine thus designated. When necessary, we call just `2-rig' an object of the minimal 2-rig doctrine.
\end{notation}
\begin{example}\label{ex_2rigs}
	The following are examples of 2-rigs:
	\begin{enumtag}{ra}
		\item \label{re_1} Any monoidal category $(\clV,\otimes,I)$ with the property that $\otimes$ preserves $\kappa$-ary coproducts is a monoidally $\kappa$-additive category. This includes the category of sets, any cartesian closed category with finite coproducts, the category of modules over a ring $R$ or, more generally, the category $\Mod^{\clV}_R$ of modules over a monoid $R$ in a suitable monoidal base $\clV$.
		\item \label{re_2} In the same notation, the category $[\clA,\clV]$ of $\clV$\hyp{}enriched presheaves over a (symmetric) monoidal $\clV$\hyp{}category $(\clA,\oplus,j)$, endowed with the Day convolution monoidal structure
		is a (symmetric) closed 2-rig.
		\item An example of a non-symmetric 2-rig is the category $[\clA,\clA]_+\subseteq [\clA,\clA]$ of endofunctors $F : \clA \to \clA$ that commute with finite coproducts.
	\end{enumtag}
\end{example}
\section{Modules and strengths}\label{modules_n_stre}
Just as ordinary rings and rigs act on modules, so 2-rigs or $\bsD$-rigs (for a 2-rig doctrine $\bsD = (\bsA, \bsM, \delta)$) act on 2-modules, sometimes called \emph{actegories} (cf. \cite{janelidze2001note}). For the same additive doctrine $\bsA$, if
$\clC$ is an $\bsA$-algebra, then we may form the endohom $[\clC, \clC]$ of $\bsA$-algebra maps or $\bsA$-cocontinuous functors, and this endohom forms a monoidally $\bsA$-cocomplete category. If in addition $\clC$ is a $\bsD$-rig, then it has an underlying monoidally $\bsA$-cocomplete category.

Keeping this in mind, we give the following definition to capture an action of $\clR$ on a category $\clC$ as a suitable rig endomorphism.
\begin{definition}
	A (left) $\clR$-module structure (or actegory structure) on $\clC$ is a monoidally $\bsA$-cocontinuous map
		$\clR \to [\clC, \clC]$.
\end{definition}
\begin{example}\label{cailei}
A simple example is $\clR$ acting on itself, so the map above takes an object $R$ to the functor $R \otimes -: \clR \to \clR$. This is called the \emph{left Cayley action}. For objects $R$ of $\clR$ and $C$ of $\clC$, we sometimes use $R \lact C$ to denote values of left module actions. In some tautological cases, for example, the left Cayley action, we use ordinary tensor product notation $R \otimes R'$.
\end{example}
\begin{remark}
As a monoidal category, $\clR$ may also be construed as a one-object bicategory $B\clR$, and an $\clR$-module may be construed as a pseudofunctor of bicategories
	$B\clR \to \bsA\emdash\cate{Alg}$
that is locally $\bsA$-cocontinuous.
\end{remark}
In this notation, we can provide a sensible notion for a morphism of modules.
\begin{definition}
	Given $\clR$-modules $\clC, \clD: B\clR \to \bsA\emdash\cate{Alg}$, a \emph{morphism} from $\clC$ to $\clD$ is a lax natural transformation $\clC \to \clD$.
\end{definition}
It is worth unpacking this very terse definition. Here a lax natural transformation takes the unique object of $B\clR$ to a 1-cell $F: \clC \to \clD$, in other words an $\bsA$-continuous functor of this form. It takes 1-cells of $B\clR$, i.e. objects $R$ of $\clR$, to 2-cells which take the form of families in $\clD$,
\[\label{dis}
	R \lact FC \to F(R \lact C),
\]
that are natural in $C$. This 2-cell constraint is often called a \emph{strength} on $F$; we call it a \emph{left strength}. The lax naturality axioms provide the usual axioms for a tensorial strength as defined, e.g. in \cite{kock1972strong}.

One can define right module structures by reversing the 1-cells of $B\clR$, i.e., reversing the order of tensoring,
	$(B\clR)^\op \to \bsA\emdash\cate{Alg}$.
For example, we have a right Cayley action that takes an object $R$ to $- \otimes R$. Then, a 2-cell constraint for a lax natural transformation between right module structures is called a \emph{right strength}. It involves natural families, sometimes written as
\[\label{dat}
	FC \ract R \to F(C \ract R).
\]
Similarly, one can define bimodules as homomorphisms
	$(B\clR)^\op \times B\clR \to [\clC, \clC]$
(for example, there is an evident Cayley bimodule with $\clR$ acting on itself on both the left and right), and consider bistrengths.
\begin{example}
	Here is one type of example that recurs frequently for us. Suppose given a $\bsD$-rig map $\varphi: \clR \to \clS$. This induces a homomorphism $\varphi^\op \times \varphi: B\clR^\op \times B\clR \to B\clS^\op \times B\clS$, which composes with the Cayley bimodule of $\clS$. Letting  $\alpha_{\clR}$, $\alpha_{\clS}$ denote the Cayley bimodules, the data of a morphism from $\alpha_{\clR}$ to $\alpha_{\clS}(\varphi^\op \times \varphi)$ entails an $\bsA$-cocontinuous functor $G: \clR \to \clS$ with a (`$\varphi$-augmented') left and right strength $\varphi(R) \otimes GR' \to G(R \otimes R')$ and $GR' \otimes \varphi(R) \to G(R' \otimes R)$.
\end{example}
\section{Differential 2-rigs: basic theory}\label{diff_2_basic}
We now turn to the main definition of the present paper, that of a \emph{derivation} on a 2-rig; in simple terms, if a 2-rig categorifies the notion of ri(n)g $R$, a derivation on a 2-rig categorifies the notion of derivation on $R$, widely used in commutative algebra and finding applications to the Galois theory of differential equations (\cf \cite{pommaret1994differential,van2012galois}).
\begin{definition}[Derivation on a 2-rig]\label{def_derivation}
	Let $\bsD$ be a 2-rig doctrine, let $\clR$ be a $\bsD$-rig, and let $\clM$ be a $\clR$-bimodule. An $\clM$-valued \emph{derivation} of $\clR$ is a bimodule morphism $\d$ from the Cayley bimodule of $\clR$ (cf. \ref{cailei}) to $\clM$, such that the canonical natural maps
	\[\fkl : \d C \ract C' \cup C\lact \d C' \to \d(C\otimes C') \qquad \fki : 0 \to \d(I)\]
	are isomorphisms.
\end{definition}
  We will refer to the map $\fkl$ above as the \emph{leibnizator} map of the derivation.

	Here the first arrow is defined by pairing the module right strength $\d C \ract C' \to \d(C \otimes C')$ with the module left strength $C \lact \d C' \to \d(C \otimes C')$.
\begin{definition}[Differential $\bsD$-rig]\label{the_main_def}
	A \emph{differential $\bsD$-rig} is a $\bsD$-rig $\clR$ equipped with a derivation from the Cayley bimodule of $\clR$ to itself.
\end{definition}
\begin{example}\label{spe_exa}
A paradigmatic example of a differential 2-rig is given by the category of Joyal species with its standard derivative functor, sending $F : \Sigma^\op \to \Set : n\mapsto Fn$ to $F' : n\mapsto F(n+1)$, where $n\in\Sigma$ is an $n$-element set.
\end{example}
In this example, the doctrine is that of symmetric monoidally cocomplete categories, and the category of species is the free symmetric monoidally cocomplete category on one object, i.e. the category of finite sets and bijections. This is the
category of presheaves $\Sigma^\op \to \Set$ on the category of finite sets and bijections, equipped with the Day convolution product induced from the monoidal product on $\Sigma$.

Besides the Leibniz rule, whose validity can be proved via elementary methods, the differential structure in the category of species satisfies two additional properties reminiscent of formal power series theory.\footnote{Given the elementary nature of their proof, we believe both these results pertain to `folklore' in circles of combinatorialists, but we could not find an appropriate reference for them.} If $S,T,U,V$ are objects of $\Sigma$, we can prove the following result (cf. \cite[§8.11]{am2} and \cite[§4.5.4]{Yorgey}). (We provide a proof  of this and of \ref{taylor} appear in Appendix B, page \pageref{proof_of_iter_lei_spec}.)
\begin{proposition}[Generalised Leibniz rule for species]\label{iter_lei_spec}
	Let $\d$ be the standard derivation on species. We can think of the $n$-th derivative $\d^n F$ as a derivative `with respect to a $n$-element set $U$', since in case $|U|=n$ one has
		$\d^nF[A] = F[A+n] \cong F[A\cup U]$.
	Define $F^{(U)}$ by the formula $F^{(U)}[A] = F[A\cup U]$. Now, let $F,G : \Sigma \to \Set$ be two combinatorial species; we have
	\[
		(F\ast G)^{(U)}[C] \cong \sum_{S\cup T=U}(F^{(S)} * G^{(T)})[C].
	\]
\end{proposition}
\begin{theorem}[A Taylor-Maclaurin formula for species]\label{taylor}
	Every species $F : \Sigma \to \Set$ has a `Taylor-Maclaurin' expansion
	\[\label{tmclau}
	F(X+A)\cong
	\int^n F(A+n)\times X^n\cong
	\int^{n\in \bf P} \d^n F(A) \times X^n.
	\]
\end{theorem}
	The name of this result is motivated by the fact that when the coend in \eqref{tmclau} is unwound, we end up with the Taylor expansion $F(X+A)\cong \sum_{n=0}^\infty \frac{\d^n F(A)}{n!}X^n$.

There is a notion of morphism of differential 2-rig, and a notion of morphism of derivations: together, these define the category $2\emdash\cate{Rig}$ of differential 2-rigs, and the category $\Der(\clR,\d)$ of derivations on a given 2-rig. We will not investigate 2-categorical properties of $2\emdash\cate{Rig}$, but the notion of morphism of derivation is necessary to turn \ref{kaler_are_real} into an equivalence of categories instead of just a correspondence on objects.
\begin{definition}[Morphism of differential 2-rigs]\label{mororig}
	Given differential 2-rigs $(\clR,\d)\to (\clS,\d')$, morphisms of differential 2-rigs are morphisms of 2-rigs $F : \clR \to \clS$
	such that $\d' \circ F = F\circ\d$.
\end{definition}
\begin{definition}[Morphism of derivations]\label{moroder}
	Let $\clR$ be a 2-rig, and $\d,\d' : \clR \to \clR$ two derivations in the sense of \ref{def_derivation}. A \emph{morphism of derivations} $\alpha : \d \To \d'$ is a natural transformation of functors
	such that
	the equality of 2-cells
		$\fkl'\circ (\alpha\otimes 1 \cup 1 \otimes\alpha) = (\alpha * \otimes)\circ \fkl$
	holds if $\fkl$ (resp., $\fkl'$) is the leibnizator of $\d$ (resp., $\d'$).
\end{definition}
Now we observe how some notions bearing on 2-rigs, particularly property-like notions for the multiplicative monoidal product, make sense independent of which doctrine of 2-rigs is considered. For example, a 2-rig (relative to any 2-rig doctrine) is \emph{cartesian} if its multiplicative monoidal product is cartesian and is \emph{closed} if tensoring with an object on either side has a right adjoint.

A derivation $\d : \clR \to \clR$ is \emph{trivial} if it is constantly $0$. A $\d$-\emph{constant} object is such that $\d X\cong 0$. Clearly, a derivation is trivial if and only if every object is a $\d$-constant.
The description of derivations on a 2-rig in terms of tensorial strengths leads to two fundamental `no-go theorem' for derivations on a 2-rig: the proofs appear one after the other in Appendix B, page \pageref{proof_of_nogo1}.
\begin{proposition}\label{nogo1}
	A derivation $\d : \clR \to \clR$ on a cartesian 2-rig must be trivial.
\end{proposition}
\begin{proposition}\label{nogo2}
	Suppose that $\clR$ is a closed 2-rig and that the functor $\clR(I, -): \clR \to \cate{Set}$
	is faithful. Then any functor $\d : \clR \to \clR$ can carry at most one left strength and one right strength.
\end{proposition}
\begin{remark}\label{dim_of_rigs}
	The intuition behind \ref{nogo1} and \ref{nogo2} is that in certain categories, every object arises as a coequalizer of maps between coproducts of copies of $I$. If derivations preserve colimits and take $I$ to $0$, then, of course, every object maps to $0$.

	This allows for yet another analogy with differential/algebraic geometry:
	categories satisfying \ref{nogo1}, \ref{nogo2} are `categories of constants' hence are `$0$-dimensional' from the point of view of categorified `dimension theory'.
\end{remark}
The connection between derivations on 2-rigs and tensorial strengths deserves to be spelt out more explicitly: to this end, we provide a general procedure to turn every endofunctor $F : \clR \to \clR$ on a 2-rig into a derivation.

\subsection*{The universal construction of tensorial strengths}

This subsection shows that to every endofunctor $F : \clA \to \clA$ one can associate another endofuctor $\Theta F$, carrying the structure of a cofree coalgebra for a comonad $\Theta$ on $[\clA,\clA]$. This, in turn, follows from the fact that there is a comonad $\Theta^{\l}$ (resp., $\Theta^{\r}$) on the category $[\clA,\clA]$ of endofunctors of a monoidal category $\clA$, that equips an endofunctor $F$ with a cofree left (resp., right) tensorial strength. The proof appears in Appendix B, page \pageref{proof_theta}.
\begin{proposition}\label{the_theta_comonad}
	Let $\clA$ be a complete and left (resp., right) symmetric monoidal closed category; then, there exists a comonad $\Theta$
	on the category $[\clA,\clA]$ of endofunctors of $\clA$, whose coalgebras are exactly the endofunctors equipped with a right (resp., left) tensorial strength.
\end{proposition}
\begin{remark}
	This result entails that given an endofunctor $F : \clR \to \clR$ on a symmetric 2-rig, $F$ is `best\hyp{}approximated' by an endofunctor $\Theta F$ that we might think as a `lax derivation' (by this we mean a fairly weak concept: it's a functor equipped with both a left and right tensorial strength, and as a consequence with a noninvertible leibnizator map $\cop{t^{\l}}{t^{\r}} : \Theta FA\otimes B \cup A\otimes \Theta FB  \to F(A\otimes B)$), obtained by endowing the functor  $F$ it with the cofree strength
	\[
		\label{best_der}
		\xymatrix{
		\Theta F A \otimes B \ar[r]^-{t^{\l}_{AB}} & \Theta F(A\otimes B) & \ar[l]_-{t^{\r}_{AB}} A \otimes \Theta F B
		}
	\]
	using the universal property of coproducts, now one gets the desired map $\cop{t^{\l}}{t^{\r}}$.

	The result remains true when the 2-rig $\clR$ is not symmetric, but a little more care is needed; in that case, one must define $\Theta_{\r}$ (resp., $\Theta_{\l}$) exploiting a right (resp., left) closed structure on $\clA$.
\end{remark}

Given any category $\clC$ and a morphism $f : X\to Y$ one can consider the category $\clC[f^{-1}]$ obtained as the `smallest' category where $f$ becomes an isomorphism (cf. \cite[1.1]{GZ}). In this light, the relevance of this result lies in the fact that one can then formally invert the map $\cop{t^{\l}}{t^{\r}}$ above and endow $\clR$ with a derivation canonically obtained from the pair $(\clR, F)$.

We conclude the section by turning our attention to the following result, which to the best of our knowledge is new, despite the simplicity of its proof: let $M$ be an internal $\otimes$-monoid in a differential 2-rig $(\clR,\otimes, \d)$; the derivative $\d M$ is an $M$-module. The proof appears in Appendix B, page \pageref{proof_of_der_ova_monoid}.
\begin{proposition}\label{der_ova_monoid}
	Let $\clR$ be a 2-rig, and $M$ a internal semigroup (resp., monoid) in $\clR$, with multiplication $m : M\otimes M \to M$ (and unit $e : I \to M$); then the map $\d m : \d M\otimes M \cup M\otimes \d M \to \d M$ amounts to a pair of actions $i_R : \d M \otimes M \to \d M$ and $i_L : M\otimes \d M \to \d M$ of $M$ on its derivative object $\d M$.
\end{proposition}
\section{The construction of free 2-rigs}\label{free_2_rigs}
In this subsection, we turn our attention to constructions of derivations and differentials, restricting focus to symmetric 2-rig doctrines $\bsD$. Our main technique is to exploit the representability of derivations in the sense of \ref{squzero} and \ref{prop_augm}.

There are several reasons for restricting to symmetric 2-rigs $\clR$. First, in ordinary algebra, the vast majority of applications of derivations are to commutative algebra; categorifying, it is then natural to consider symmetric monoidal structures. Moreover, tensoring functors $R \otimes -: \clR \to \clR$ carry canonical (co)strengths, on account of the symmetry.

In the symmetric case, we can turn any left $\clR$-module $\clM$ into a right module or a bimodule by defining $M \ract R$ to be $R \lact M$. We call bimodules arising this way \emph{symmetric}. (Here, it seems pointless to distinguish between $\lact$ and $\ract$, so we write $\otimes$ instead.)

\begin{definition}[Square-zero extensions]\label{squzero}
	Let $\clR$ be a $\bsD$-rig, and let $\clM$ be a symmetric $\clR$-bimodule.
	Define the \emph{square-zero extension} $\clR \ltimes \clM$ of $\clM$ to be $\clR\times \clM$ as an $\bsA$-algebra, and equipped with a symmetric monoidal product defined by the formula
	\[(A,M)\boxtimes(B,N) := (A\otimes B, A\otimes N \cup M\otimes B),\]
	and with monoidal unit $(I, 0)$. The first projection $\pi: \clR \ltimes \clM \to \clR$ makes this a $\bsD$-rig over $\clR$.
\end{definition}

A straightforward computation allows determining the associators and unitors for the $\boxtimes$ monoidal structure (one must use the compatibility between the left and right module structure on $\clM$) and the left and right distributive maps.

An alternative presentation of the square zero extension, in the case where $\clM$ is the Cayley bimodule of $\clR$ acting on itself, can be given as a `quotient' 2-rig $\clR[Y]/(Y^2)$: a categorification of an algebra of `dual numbers', as explained in the following subsection. This 2-rig is denoted $\clR[\varepsilon]$.

\begin{proposition}\label{prop_augm}
	For a $\bsD$-rig $\clS$ over $\clR$, say $\psi: \clS \to \clR$, there is a natural equivalence between maps $\Phi: \clS \to \clR \ltimes \clM$ in $\bsD\emdash\cate{Rig}/\clR$, and $\psi$-augmented derivations $\d$ of $\clS$ valued in $\clM$, where $\d = \pi_2 \Phi: \clS \to \clM$.
\end{proposition}
The proof is fairly routine since $(\psi, \d)$ being a (strong) symmetric monoidal functor means that we obtain isomorphisms
	$\d(S) \otimes \psi(S') \cup \psi(S) \otimes \d(S') \cong \d(S \otimes S')$
whose restrictions to the summands satisfy the strength coherence conditions, on account of the coherence conditions that obtain for a symmetric monoidal functor.

For example, we can use this proposition to reconstruct the standard derivative on Joyal species $\Spc$, working over the doctrine $\bsD$ of symmetric monoidally cocomplete categories. Consider $\Spc$ as a Cayley bimodule over itself, and form $\Spc[\varepsilon] = \Spc \ltimes \Spc$.

As $\Spc$ is the free symmetric monoidally cocomplete category on one generator $X$ (the representable functor $\Sigma(-, 1)$), there is an equivalence of categories
\[\label{spc_is_freeof1}
	\bsD\emdash\cate{Rig}(\Spc, \Spc[\varepsilon]) \simeq \Spc[\varepsilon].
\]
This means any object $(F, G)$ whatsoever of $\Spc[\varepsilon]$ induces a $\bsD$-rig map $\Phi_{(F, G)}: \Spc \to \Spc[\varepsilon]$, hence (by the Proposition) a $\psi$-augmented derivation for some $\bsD$-rig map $\psi: \Spc \to \Spc$.
Let us be more explicit. First we calculate $\psi = \pi \Phi_{(F, G)}: \Spc \to \Spc$. The pseudonaturality of the equivalence
	$\bsD\emdash\cate{Rig}(\Spc, \clR) \simeq \clR$
means $\pi \Phi_{(F, G)}$ is the unique (\emph{essentially} unique, i.e., unique up to unique isomorphism) symmetric monoidally cocontinuous functor $\psi_F: \Spc \to \Spc$ that carries $X$ to $F$. Proceeding in stages, the functor $F: 1 \to \Spc$ extends essentially uniquely to a symmetric monoidal functor $\tilde{F}: \Sigma \to \Spc$, taking $n$ to the $n$-fold Day convolution $F^{\otimes n}$. Then this extends essentially uniquely to a \emph{cocontinuous} symmetric monoidal functor $\Spc = [\Sigma^\op, \Set] \to \Spc$, according to the formula

\[
	W = \int^{n: \Sigma} W(n) \cdot \Sigma(-, n) \;\; \mapsto \;\; \int^{n: \Sigma} W(n) \cdot F^{\otimes n}.
\]
The last coend is an instance of the \emph{substitution} product of species, denoted $W \circ F$. Whatever it is, the point is that $\psi = (-) \circ F$, where the right side is the essentially unique $\bsD$-rig map $\Spc \to \Spc$ that extends $F: 1 \to \Spc$. In particular, if $F$ is the generator $X$, then $\psi_X$ is the identity on $\Spc$.

Now a derivation $\d: \Spc \to \Spc$ augmented by the identity is just an ordinary derivation, i.e., satisfies $\d(A \otimes B) \cong \d(A) \otimes B \cup A \otimes \d(B)$. The composite
	$1 \stackrel{X}{\to} \Spc \stackrel{\langle \text{id}, \d\rangle}{\to} \Spc[\varepsilon] \stackrel{\pi_2}{\to} \Spc$
is the component $G$ of $(F, G)$, whereas the composition of the last two arrows is $\d$. In other words, $G = \d(X)$. If we want $\d$ to match the standard derivative of species, then we must have $G = X' = I$, the unit of Day convolution.

Therefore, under the natural equivalence of the proposition, the standard derivative of species corresponds to the $\bsD$-rig map $\Spc \to \Spc[\varepsilon]$ that takes the generator $X$ to $(X, I)$. If we take $X$ to some other element $(X, G)$ instead, then the corresponding derivation $\d$ is defined by $\d(F) = F' \otimes G$, because this is after all a derivation, and because $\d(X) \cong X' \otimes G \cong G$ is correct. Note then that every differential structure, i.e., every derivation on $\Spc$ augmented over the identity, is obtained by tensoring the standard derivative by some object.

\begin{remark}
In the analogy between species $F,G$ and formal power series $f,g$, the substitution product corresponds to functional substitution $(f\circ g)(x) = f(g(x))$. The derivative of a substitution can be computed via the \emph{chain rule}, known since Joyal \cite{joyal1981theorie}:
\[(F \circ G)' = (F' \circ G) \otimes G'.\]
We will provide a proof for the chain rule, formulated not only for species but valid in any $\bsD$-rig, in Appendix B, page \pageref{proof_of_chain}.
\end{remark}

\subsection*{Presentations of \texorpdfstring{$\bsD$}{D}-rigs}
Here we provide a construction of free $\bsD$-rigs and give a few sample constructions of other $\bsD$-rigs. We freely employ the definitions we have introduced so far, and in particular \ref{additive_doc}, \ref{mul_doc}. Our main result, \ref{equ_kahler}, is guided by an analogy with classical algebra: to provide a presentation of an ordinary rig is tantamount to providing a coequalizer of two maps between free rigs since rigs form a category 2-monadic over $\Set$.

The fact that under mild assumptions on $\bsD$ --for example, if its multiplicative monad $\bsM$ is finitary-- the 2-category $\bsD\emdash\cate{Rig}$ has bicolimits, ensures that similar such constructions exist and can provide presentations of 2-rigs as suitable 2-dimensional colimits \cite{2catlimits} of diagrams of free 2-rigs.

If $\bsA$ denotes the monad on $\Cat$ for the additive doctrine, then for a category $\clC$, the $\bsA$-cocompletion $\bsA(\clC)$ is equivalent to the full subcategory of the small presheaf category $P(\clC)$ obtained by taking the closure of the representable functors under the class of $\bsA$-colimits.
\begin{remark}
	Using the distributive law, the monad for $\bsD$ is the composite $\bsA\bsM$. Hence, for \emph{every} doctrine $\bsD$, the free $\bsD$-rig $\bsD[\clC]$ on a category $\clC$ is always formed according to a simple two-step procedure: first, take the free multiplicative structure generated by $C$, i.e. the category $\bsM(\clC)$. Then, take the free $\bsA$-cocompletion of $\bsM(\clC)$.
\end{remark}
We have already seen an example of this in the case of Joyal species (in the doctrine $\bsD$ of symmetric monoidally cocomplete categories): it is the free cocompletion $[\Sigma^\op, \Set]$ of the free symmetric monoidal category $\Sigma$ on a single generator. Likewise, we may define multivariate species, say for example species in two variables, as the category $[\Sigma(2)^\op, \Set]$ equipped with Day convolution, where incidentally $\Sigma(2)$ is equivalent to $\Sigma \times \Sigma$.

For the remainder of this section, we return to symmetric 2-rigs (relative to some additive doctrine $\bsA$), and proceed to categorify some commutative algebra. The 2-category of $\bsA$\hyp{}algebras, being a 2-category of algebras for a KZ-monad, carries a monoidal product $\odot$ (see \cite{franco2011pseudo}) characterized by the fact that for $\bsA$-algebras $\clA$, $\clB$, $\clC$, functors $\clA \times \clB \to \clC$ that are $\bsA$\hyp{}cocontinuous in the separate $\clA$-, $\clB$-arguments are equivalent to $\bsA$\hyp{}cocontinuous functors $\clA \odot \clB \to \clC$.
\begin{proposition}[Coproduct of $\bsD$-rigs]
	Using the universal property one can show that if $\clR$, $\clS$ are $\bsD$-rigs, meaning here symmetric monoidally $\bsA$-cocomplete categories, then $\clR \odot \clS$ naturally acquires a $\bsD$-rig structure and is the coproduct of $\clR$ and $\clS$ in $\bsD\emdash\cate{Rig}$.
\end{proposition}
\begin{notation}[Extension of scalars]\label{ext_of_scala}
In particular, let $\clS = \bsD[Y]$ be the free $\bsD$-rig on a single generator $Y$. We write $\clR \odot \bsD[Y]$ as $\clR[Y]$; this plays a role analogous to a polynomial rig $C[Y]$ with coefficients in a rig $C$, and the construction is analogous to the `extension of scalars' from the initial rig $\bbN[Y]$ to the rig $C[Y]$ obtained as a coproduct in the category of rigs.
\end{notation}
\begin{remark}
    The formation of $\clR[Y]$ does not require working with symmetric 2-rigs: just as one can form a polynomial algebra $R[x]$ over a noncommutative rig $R$, so one can form a `polynomial' 2-rig $\clR[Y]$ over a monoidal 2-rig $\clR$, by taking a tensor product $\clR \odot \bsD[Y]$. However, this tensor product will generally not be a coproduct in $\bsD\emdash\cate{Rig}$ if we work outside the symmetric context.
\end{remark}
\subsection*{K\"ahler differentials}
Next, we sketch the construction of a $\bsD$-rig of K\"ahler differentials on a $\bsD$-rig $\clR$. Again, we borrow ideas from the analogous construction in algebraic geometry. Let $\bsD[Y]$ be the free $\bsD$-rig on a single generator $\{Y\}$, treated as a generic `indeterminate'.

Let $0: \bsD[Y] \to \bsD[Y]$ denote the essentially unique $\bsD$-rig morphism that takes $Y$ to $0$, and similarly let $Y^2: \bsD[Y] \to \bsD[Y]$ denote the morphism that takes $Y$ to $Y^{\otimes 2}$. The unique map $0 \to Y^2$ in $\bsD[Y]$ transports across the equivalence
\[\label{important_equcat}
\bsD\text{-Rig}(\bsD[Y], \bsD[Y]) \simeq \bsD[Y]
\]
to a symmetric monoidal natural transformation $0 \Rightarrow Y^2$ between $\bsD$-rig maps $0, Y^2: \bsD[Y] \to \bsD[Y]$. Extending scalars like in \ref{ext_of_scala}, we obtain a 2-cell in $\bsD\emdash\cate{Rig}$:
\[
	\label{a_coequ}
	\vcenter{\xymatrix@C=1.5cm{
	\clR[Y] \rtwocell^{\initial}_{Y^2}{}
	& \clR[Y] \ar[r]^-{q} & \clR[Y]/(Y^2)
	}}
\]
The `quotient' construction $q: \clR[Y] \to \clR[Y]/(Y^2)$ we are after is a coinverter of this 2-cell in the 2-category $\bsD\emdash\cate{Rig}$. In fact, diagram \eqref{a_coequ} satisfies precisely the universal property of a \emph{coinverter} (\cite[dual of (4.6)]{2catlimits}) when we adopt for $\clR[Y]/(Y^2)$ the concrete model deduced from \ref{true_polinomi}: each object of $\clR[Y]/(Y^2)$ is of the form $A + B\otimes Y$ we prove this in Appendix B, page \pageref{proof_of_coinverting}.
\begin{remark}
	Observe that for some 2-rig doctrines $\bsD$, this coinverter may be somewhat degenerate. For example, in the doctrine of cartesian 2-rigs (for any additive doctrine $\bsA$), the condition that an arrow $0 \to C^2$ is invertible in $\clR$ forces $C \cong 0$ (because $C$ is a retract of $C^2$), and in this case, the coinverter will be the 2-rig map $\clR[Y] \to \clR$ taking $Y$ to $0$ (cf. the fact that there are no nontrivial differentials on a cartesian 2-rig).
\end{remark}
\begin{proposition}\label{equ_kahler}
    For a doctrine $\bsD$ of symmetric 2-rigs, there is an equivalence $\clR[Y]/(Y^2) \simeq \clR \ltimes \clR$.
\end{proposition}
In combination with \ref{prop_augm}, this means that $\clR[\varepsilon] = \clR[Y]/(Y^2)$, equipped with the evident $\bsD$-rig map $\clR[\varepsilon] \to \clR$ taking $Y$ to $0$, represents augmented derivations.
\begin{corollary}\label{kaler_are_real}
	There is an equivalence of categories
	\[\Der(\clR,\clR) \cong 2\emdash\cate{Rig}(\clR,\clR[Y]/(Y^2))\]
	or in other words, the category of derivations $\clR \to \clR$ as in \ref{moroder} correspond to 2-rig morphisms $\clR \to \clR[Y]/(Y^2)$.
	More generally, there is an equivalence between derivations $\clR \to \clM$ values in a $\clR$-module $\clM$, and algebra morphisms between $\clR$ and the square-zero extension of \ref{squzero}.
\end{corollary}
The construction of free $\bsD$-rigs and \ref{kaler_are_real} allow to provide examples of differentials on categories of multivariate (or `colored', cf. \cite{mendez1993colored}) species.
\begin{definition}[Partial derivative]
    Let $\bsD[S]$ be the free $\bsD$-rig on a set or discrete category of generators $S$. For $s \in S$, define the \emph{partial derivative}
    \[
    \frac{\d}{\d s}: \bsD[S] \to \bsD[S]
    \]
    to be the derivation that corresponds to the $\bsD$-rig map $\bsD[S] \to \bsD[S][\varepsilon]$ that takes $s$ to $(s, I)$ and $t \in S$, $t \neq s$, to $(t, 0)$.\footnote{One can prove that the `Schwarz-Clairaut's theorem' of commutativity of composition of derivatives with respect different `indeterminates'. We refrain to provide such a proof in detail, as it is completely straightforward.}
\end{definition}
Every differential on $\bsD[S]$ is similarly formed from the $\bsD[S]$-rig maps $\bsD[S] \to \bsD[S][\varepsilon]$ taking each $s$ to $(s, a_s)$ for some choice of `coefficients' $a_s \in \bsD[S]$. In the case where the additive doctrine admits arbitrary coproducts, this differential may be denoted
\[
\d = \sum_{s \in S} a_s \frac{\d}{\d s}.
\]

Here is one more example of a differential 2-rig, bearing witness that differential structures on a symmetric 2-rig tend to be plentiful. The idea goes as follows: let $\bsD[X, Y]$ be the free $\bsD$-rig over two generators; given any two polynomials $p(X,Y), q(X,Y)$ we can build the `quotient 2-rig' killing off the `ideal' generated by $\{p,q\}$ as a suitable 2-colimit.
\begin{example}\label{hyperbolic}
    We consider the 2-rig $\clH := \bsD[X, Y]/(Y^2 + 1 \cong X^2)$ where we categorify the coordinate ring of an hyperbola. Here we have two morphisms $\bsD[T] \to \bsD[X, Y]$ to the free $\bsD$-rig on two generators, one taking $T$ to $Y^2 + 1$, the other taking $T$ to $X^2$; to form $\bsD[X, Y]/(Y^2 + 1 \cong X^2)$, construct a co-iso-inserter (\cite{2catlimits,bourkethesis}) between these two $\bsD$-rig maps.

    The differential $\d: \clH \to \clH$ is defined by $\d(X) = Y$, $\d(Y) = X$, and taking the co-iso-inserter $\varphi: Y^2 + 1 \to X^2$ to a canonical isomorphism $\d(Y^2 \cup 1) \to \partial(X^2)$ obtained as follows:
		\begin{multline}
			\d(Y^2 \cup 1) \cong \d(Y^2) + \d(1) \cong \d (Y^2) \cong \d Y \otimes Y \cup Y \otimes \d Y\\
			\cong X \otimes Y \cup Y \otimes X \stackrel{\sigma \cup \sigma}{\to} Y \otimes X \cup X \otimes Y = \d X \otimes X \cup X \otimes \d X \cong \partial(X^2)
		\end{multline}
    where $\sigma$ denotes an instance of the symmetry isomorphism.
\end{example}
\begin{proposition}[Free 2-rigs are differential]\label{free_are_diff}
	The free 2-rig $\Sigma[Y]$ and its cocompletion $\Sigma\llbracket Y\rrbracket$ with respect to arbitrary coproducts both admit at least one nontrivial derivation, which is uniquely determined by the request that the `generator' $Y$ goes to the monoidal (Day convolution) unit.
\end{proposition}
From the universal property of $\clR[Y]$, we deduce that it is the category generated under coproducts by formal expressions $A_n \otimes Y^n$  where $n\ge 0$ is an integer and $A_n\in\clR$.
\begin{proposition}\label{true_polinomi}
	Every object in the differential 2-rig $\clR[Y]$ admits a unique representation as a formal sum like $\sum_{i=0}^d A_i\otimes Y^i$.
\end{proposition}
\begin{proof}
	In Appendix B, page \pageref{proof_of_true_polinomi}.
\end{proof}
A particularly interesting example of a free 2-rig construction as differential 2-rig is where $S$ is a countable set whose elements we denote $\{Y, Y^{(1)},Y^{(2)}\dots, Y^{(n)},\dots\}$, that we interpret as the stock of all subsequent derivatives of a unique indeterminate $Y$. In other words, we construct a differential $\d: \bsD[S] \to \bsD[S]$ via the $\bsD$-rig map
\[
\bsD[S] \to \bsD[S][\varepsilon]
\]
that takes $Y^{(i)}$ to $(Y^{(i)}, Y^{(i+1)})$, in effect defining $\d(Y^{(i)}) = Y^{(i+1)}$.
This construction has a parallel in differential algebra, see e.g. \cite[Ch. 1]{van2012galois}. Hence we obtain, by `scalar extension' (tensoring with $\clR$)
\begin{example}[The 2-rig of differential polynomials]\label{diff_polys}
	We can define the \emph{2-rig of differential polynomials (with coefficients in a 2-rig $\clR$)} using an infinite set of `indeterminates' $\clY := \{Y=Y^{(0)}, Y^{(1)},Y^{(2)}\dots, Y^{(n)},\dots\}$ as above, and defining the 2-rig $\clR[Y^\d]$ as the free 2-rig of polynomials over $\clY$. This is a differential 2-rig where the differential $\d$ takes every `constant' $C \odot I \in \clR \odot \bsD[\clY]$ to $0$, and $\d(Y^{(i)})$ to $Y^{(i+1)}$.
\end{example}
The 2-rig $\bsD[Y^\d]$ defined above enjoys the following universal property: given a differential $\bsD$-rig $\clS$ and an element $A\in \clS$, there exists a unique morphism of differential 2-rigs $\bar X : \Sigma[Y^\d] \to \clS$ with the property that $Y\mapsto A$. In other words,
\begin{theorem}\label{diff_poly_rig}
	The free $\bsD$-rig of polynomials $\Sigma[Y^\d]$ of \ref{diff_polys} is the free differential 2-rig on a single generator $\{Y\}$.
\end{theorem}
\begin{remark}
    A slightly different way to put this theorem is: the monad $\bfE$ on $\Cat$, whose algebras are categories $\clR$ equipped with an endofunctor $D: \clR \to \clR$, distributes over the 2-rig monad $\bsA\bsM$ according to the Leibniz rule.\footnote{Intuitively, treat $D$ as a differential operator so that $D$ applied to a polynomial operator can be rewritten as a polynomial operator applied to $D$.} If $\bsD$ is a symmetric 2-rig doctrine, then the free differential $\bsD$-rig on a set of generators $S$ is
    $\bigodot_{s \in S} \bsD[Y_s^{\d}] = \bsD[\{Y_s^{(i)}\}_{s \in S, i \in \bbN}]$.
\end{remark}
We conclude the section concentrating on the proof of a \emph{chain rule} on free $\bsD$-rigs. If, following \cite{Joyal1986foncteurs}, we shall think about combinatorial species as categorified formal power series, a `chain rule' of the form $(f\circ g)'(x) = f'(g(x))g'(x)$ shall hold; it follows from an easy computation that this is the case when the substitution $F\circ G$ is interpreted as a \emph{substitution product}
(cf. for example \cite[§1.4]{bergeron1998combinatorial}). The present subsection provides a conceptual argument proving a chain rule valid for an abstract symmetric 2-rig doctrine.

Let $\bsD$ be a symmetric 2-rig doctrine, and recall equation \eqref{important_equcat}.
To each object $G$ of $\bsD[1]$, there is a corresponding $\bsD$-rig map denoted $- \circ G: \bsD[1] \to \bsD[1]$. Indeed, endofunctor composition on the left side $\bsD\text{-Rig}(\bsD[1], \bsD[1])$ transports to a monoidal structure on $\bsD[1]$ which, by abuse of notation, we denote as
$\circ: \bsD[1] \times \bsD[1] \to \bsD[1]$;
variously called the \emph{substitution} monoidal product or \emph{ple\-thys\-tic}  monoidal pro\-duct \cite{mendez1993colored}. The unit for the substitution product is the generator $X: 1 \to \bsD[1]$.

The standard derivative $\d: \bsD[1] \to \bsD[1]$ is defined by $\d(X) = I$, i.e., is given by the unique $\bsD$-rig map $\bsD[1] \to \bsD[1][\varepsilon]$ that takes $X$ to $(X, I)$. The proof of the chain rule appears in Appendix B, page \pageref{proof_of_chain}.

\begin{theorem}\label{chain} Given species $F,G$, there is a canonical isomorphism
$(F \circ G)' = (F' \circ G) \otimes G'$.
\end{theorem}
\section{Conclusions and future work}\label{conclu}
We introduced the notion of differential 2-rig as a unifying structure for many diverse instances of a category equipped with a `derivation', an endofunctor that satisfies the Leibniz property.

The link between the Leibniz property for an endofunctor and a pair of tensorial strengths thereon hints at a connection between differential structures and \emph{applicative} structures, wi\-de\-ly used in functional programming \cite{McBride2008ApplicativePW,Paterson2012ConstructingAF}. Given the `geometric' flavour of differential 2-rig theory, this is a surprising connection between apparently disconnected fields that will be further investigated.

Another enticing future direction of investigation in\-vol\-ves \emph{differential equations}: one can define a `differential polynomial endofunctor' (DPE) in a similar fashion in which polynomial functors are defined inductively (cf. \cite[§2.2]{jacobs2017introduction}), by declaring that all polynomial expressions $\sum_{i=0}^n A_i\otimes \d^i$ obtained from a differential $\d : \clR \to \clR$ on a 2-rig form the category $\cate{DPE}(\clR,\d)$. The theory of differential equations in the category of species has a long and well-established history: it was most\-ly developed by Leroux and Viennot \cite{0cbbef7b84b3481ef76c8e88a86f80d411b7492f,4871af04f741391c618dddc08247a3c09b1c707d,5ec089353e113ab36e728b8df6401b98de7a5d1c,0445243bd75f64484c47f7db18f2569031b5e3bd} Labelle \cite{labelle1985combinatoire} and other authors built on that \cite{5cec315a6425164172c54ba1f2826e3856592949,a7260d9aa6cf11af5e8a6e16d6cb29f836d2ff04}. The general theory of combinatorial differential equations studied in these papers might fruitfully be framed into a more general theory of DPEs and their solutions.

\printbibliography
\appendix
\section{Coherence conditions for strengths}

\begin{definition}[Morphism of $\clR$-modules]\label{moromod}
	Given a monoidal 2-category, let $\clR$ be a pseudomonoid, and let $\clM,\clN$ be two $\clR$-bimodules. We denote the left and right unit constraints by $j$ and $k$, and left and right associativity constraints by $\alpha$ and $\beta$.
	A (lax) \emph{morphism of $\clR$-bimodules} $F : \clM \to \clN$ is a 1-cell $\clM\to\clN$, together with, 2-naturally in objects $\clA$, maps
	\[
		\vcenter{\xymatrix@R=0cm{
		C\lact FM \ar[r]^-{\xi^{\l}} & F(C\lact M) \\
		FM\ract C' \ar[r]^-{\xi^{\r}} & F(M\ract C')
		}}
	\]
	for every $C,C': \clA \to \clR$ and $M: \clA \to \clM$.

	These maps must satisfy the following coherence conditions (we give only the ones pertaining to the left constraints $\lambda, \alpha$):
	\begin{itemize}
		\item naturality in both components; the diagrams
		      \[\vcenter{\xymatrix{
			      F(C\lact M) \ar@{<-}[r]^-{\xi^{\l}}\ar[d]_{F(f\lact u)} & C \lact FM \ar[d]^{f\lact Fu}\\
			      F(C'\lact M') \ar@{<-}[r]_-{\xi^{\l}} & C' \lact FM'\\
			      }}\]
		      are commutative, for every pair of morphisms $f : C\to C'$ and $u : M\to M'$.
		\item compatibility with the monoidality of the action maps, in the form of compatibility with the isomorphisms $C\lact (C'\lact M)\cong (C\otimes C')\lact M$ witnessing the strong monoidality of the action functor and $I\lact M\cong M$: the diagram
		      \begin{gather*}
			      \vcenter{\xymatrix{
			      &F(I\lact M)\ar[dr]^{Fj}\ar[dl]_{\xi^{\l}}& \\
			      I\lact FM \ar[rr]_j && FM
			      }}\\
			      \vcenter{\xymatrix{
			      F((C\otimes C')\lact M) \ar[r]^{F\alpha}\ar@{<-}[d]_{\xi^{\l}}& F(C\lact (C'\lact M))\ar@{<-}[d]^{\xi^{\l}} \\
			      (C\otimes C')\lact FM \ar[d]_\alpha & C\lact F(C'\lact M)\ar@{<-}[d]^{\xi^{\l}} \\
			      C\lact (C'\lact FM) \ar@{=}[r] & C\lact (C'\lact FM)
			      }}
		      \end{gather*}
		      are commutative, for $C, C'\in\clR$, $M\in\clM$.
	\end{itemize}
\end{definition}
\section{Proofs}
\begin{proof}[Proof of \ref{nogo1}]\label{proof_of_nogo1}
	If $X$ is an object of $\clR$, then we have a map $\d (!): \d (X) \to \d (1) = 0$. But for any object $A$ that admits a map $f: A \to 0$, we must have $A \cong 0$, because the composite $\pi_2 \circ (f, 1_A) : A\to 0\times A \to A$
	is the identity of $A$, and $0 \times A \cong 0$ by distributivity.
\end{proof}
\begin{proof}[Proof of \ref{nogo2}]\label{proof_of_nogo2}
	Let $[A, -]$ be the right adjoint of $A \otimes -: \clR \to \clR$. Then left strengths on $T$ are in natural bijection with enrichment structures on $T$, i.e. maps $t_{AB} : [A, B] \to [TA, TB]$,
	and by application of the faithful functor $\clR(I, -): \clR \to \cate{Set}$, such enrichment structures map one-to-one (not onto necessarily) to $\cate{Set}$-enrichment structures $\clR(A, B) \to \clR(TA, TB)$. However, there is only one of these.
\end{proof}
\begin{proof}[Proof of \ref{der_ova_monoid}]\label{proof_of_der_ova_monoid}
	Let $m : M\otimes M \to M$ be the multiplication of $M$; the map $\d m$ is of the following form
	\[
		\leib{M}{M} \xto{\d m} \d M
	\]
	and by the universal property of coproducts, it can be written as the map $\cop{i_R}{i_L}$, where
	\[
		i_R : \d M \otimes M \to \d M \qquad i_L : M \otimes \d M \to \d M.
	\]
	Evidently, these maps are our candidate right and left actions of $M$ over $\d M$.

	Now, the fact that $m$ is associative is witnessed by the commutative square
	\[
		\vcenter{\xymatrix{
		M\otimes M \otimes M\ar[r]^-{M\otimes m}\ar[d]_{m\otimes M} & M \otimes M \ar[d]^m\\
		M \otimes M \ar[r]_m & M
		}}
	\]
	If we derive it, applying $\d$ to each map, we get the commutative square
	\begin{adju}[.475]
		\xymatrix@C=-2cm{
		& \d M \otimes M \otimes M \cup M\otimes \d M \otimes M \cup M \otimes M \otimes \d M \ar[dr]^{\qquad\d M \otimes m \cup M\otimes \d m}\ar[dl]_{\d m \otimes M \cup m \otimes \d M\qquad} & \\
		\leib{M}{M} \ar[dr]_{\cop{i_R}{i_L}} && \leib{M}{M} \ar[dl]^{\cop{i_R}{i_L}}\\
		& \d M
		}
	\end{adju}
	which, thanks to the Leibniz action of $\d$ on morphisms, can be seen as the object- and morphism\hyp{}wise sum of two diagrams
	\[
		\vcenter{\xymatrix{
				\ar[r]^{\d M\otimes m}\ar[d]_{\d m \otimes M} & \ar[d]^{i_R} &&&
				\ar[r]^{M\otimes \d m}\ar[d]_{m \otimes \d M} & \ar[d]^{i_L} \\
				\ar[r]_{i_R} & &&&
				\ar[r]_{i_L} &
			}}
	\]
	witnessing precisely that $i_R$ is a right action, and $i_L$ is a left $M$-action on $\d M$.
\end{proof}
\begin{proof}[Proof of \ref{true_polinomi}]\label{proof_of_true_polinomi}
	The proof is divided into two parts: first, we show that every object of $\Sigma[Y]$ can be written as a formal sum $\sum E_i\cdot Y^i$, where $E_i$ is a set and $Y^i$ is the $i$th convolution power of the monoidal unit for Day convolution; then, we show that similarly, every object of $\clR[Y]$ can be written as $\sum A_i\otimes Y^i$.

	As for the first claim, it follows from the fact that $\Sigma[Y]$ is the closure under coproducts of representables; as for the second claim, we shall show that $\clR[Y]$ has the universal property of the coproduct $\clR \odot \Sigma[Y]$, or more clearly, the pushout of the span $\Sigma[Y] \leftarrow \Fin \to \clR$.\footnote{Something analogous happens in commutative algebra, where rings of polynomials with coefficients in $R$ can be defines from free $\bfZ$-algebras $\bfZ[X]$ via a universal property, that the diagram
	\[\vcenter{\xymatrix{
				\bfZ \po \ar[r]\ar[d] & R \ar[d]\\
				\bfZ[Y] \ar[r] & R[Y]
			}}\]
	is a pushout.}

	Inspecting the universal property: first of all, there is an obvious cospan of 2-rig morphisms $\clR \to \clR[Y] \leftarrow \Sigma[Y]$ sending $C$ to $C\otimes Y^0$ and $[n]$ to $Y^n$; and given a diagram
	\[\vcenter{\xymatrix{
		& \clR \ar[d]\ar@/^1pc/[ddr]^G\\
		\Sigma[Y]\ar[r] \ar@/_1pc/[drr]_F & \clR[Y]\ar@{.>}[dr]|{\cop{F}{G}} & \\
		&& \clB
		}}\]
	we can define a unique dotted functor $\cop{F}{G} : \clR[Y] \to \clB$ as
	\[\sum_{i=0}^d A_i\otimes Y^i \mapsto \sum_{i=0}^d GA_i\otimes (FY)^{\otimes n}, \]
	since a 2-rig morphism $F : \Sigma[Y] \to \clB$ is completely determined by the image of $Y=y(1)$.
\end{proof}
\begin{proof}[Proof of \ref{the_theta_comonad}]\label{proof_theta}
	Let's examine the left closed case: this means that every $A\otimes \firstblank$ has a right adjoint. The right closed case is analogous, mutatis mutandis.

	The condition of having a right tensorial strength amounts to the presence of maps $t_{AB} : A\otimes DB \to D(A\otimes B)$ satisfying suitable conditions.

	The maps $t_{AB}$ now transpose to
	\[
		\xymatrix{\hat{t}_{AB} : DB \ar[r] & [A,D(A\otimes B)]}
	\]
	and the $\hat{t}_{AB}$'s are natural in $B$, and a wedge in $A$: this means that there is a unique map
	\[
		\xymatrix{\hat{t}_{B} : DB \ar[r] & \displaystyle\int_A[A,D(A\otimes B)];}
	\]
	we now claim that
	\begin{enumtag}{m}
		\item the correspondence $\lambda B.\int_A[A,D(A\otimes B)]$ is an endofunctor of $\clA$;
		\item the correspondence $\Theta : D\mapsto \lambda B.\int_A[A,D(A\otimes B)]$ is an endofunctor of $[\clA,\clA]$; moreover, it is a comonad;
		\item a $\Theta$-coalgebra is exactly an endofunctor equipped with a right tensorial strength, whose components are obtained from the coalgebra map by reverse-engineering the construction of $\Theta$.
	\end{enumtag}
	The last part of the third claim is obvious; what remains of the third claim is an exercise on diagram chasing. Functoriality is evident from the canonical way in which we built $\Theta$, and $DB \to \int_A[A,D(A\otimes B)]$ attach to the components of a natural transformation $D \To \Theta(D)$.

	It remains to show that $\Theta$ is a comonad:
	\begin{itemize}
		\item the counit is obtained from the terminal wedge of $\Theta(D)$, taking the component on the monoidal unit (say, $I$):
		      \[
			      \xymatrix@C=1.4cm{ \int_A[A,D(A\otimes B)] \ar[r]^-{\epsilon_B =\pi_I} & [I, D(I\otimes B)] \cong DB  }
		      \]
		\item the comultiplication is obtained from the following computation:
		      \begin{align*}
			      \Theta \Theta (D)(A) & = \int_B [B, \Theta (D)(A \otimes B)]                       \\
			                           & \cong \int_B [B, \int_C [C, D(A \otimes B \otimes C)]]      \\
			                           & \cong \int_B \int_C [B, [C, D(A \otimes B \otimes C)]]      \\
			                           & \cong \int_B \int_C [B \otimes C, D(A \otimes B \otimes C)]
		      \end{align*}
	\end{itemize}
	It is evident, now, that the projections $\pi_{B\otimes C}$ of the terminal wedge of $\Theta(D)$ assemble into a morphism $\sigma : \Theta \To \Theta\Theta$
	of the right type; moreover, this choice of $\epsilon$ and $\sigma$ is the unique that satisfies the counit equations of a comonad; showing that $\sigma : \Theta \To \Theta^2$ is coassociative is a matter of diagram chasing.
\end{proof}
\begin{proof}[Proof of \ref{equ_kahler}]\label{proof_of_equ_kahler}
    Let $\bsD[Y] \to \clR \ltimes \clR$ be the essentially unique $\bsD$-rig map that takes $Y$ to $(0, I)$, and let $\clR \to \clR \ltimes \clR$ be the map taking $C$ to $(C, 0)$. By pairing these maps, we get a map $\varepsilon: \clR[Y] \to \clR \ltimes \clR$ out of the coproduct $\clR[Y] = \clR \odot \bsD[Y]$. It is clear that $\varepsilon$ coinverts the 2-cell $0 \Rightarrow Y^2$. Given a $\bsD$-rig map $F: \clR[Y] \to \clS$ that coinverts this 2-cell, define a map $\overline{F}: \clR \ltimes \clR \to \clS$ that takes $(R, 0)$ to $F(R)$, and $(0, I)$ to $F(Y)$. One may check that $\overline{F}$ is a $\bsD$-rig map.
\end{proof}
\begin{proof}[Proof of \ref{chain}]\label{proof_of_chain}
    Let $\d$ denote the standard derivative, and denote the $\bsD$-rig map $- \circ G$ by $\varphi$. Then the left side corresponds to the value at an object $F$ of the composite $\bsD$-rig map
    \[
    \bsD[1] \stackrel{\varphi}{\to} \bsD[1] \stackrel{\langle 1, \d\rangle}{\to} \bsD[1][\varepsilon],
    \]
    taking $F$ to $(F \circ G, (F \circ G)')$ and taking $X$ to $(G, G')$. On the other hand, $\varphi \circ \d: \bsD[1] \to \bsD[1]$ is a $\varphi$-augmented derivation, and so is $(\varphi \d) \otimes G'$. By \ref{prop_augm}, it corresponds to the $\bsD$-rig map $\bsD[1] \to \bsD[1][\varepsilon]$ taking $F$ to
    \[
    (\varphi(F), (\varphi \d(F)) \otimes G') = (F \circ G, (F' \circ G) \otimes G').
    \]
    This map is uniquely determined by where it sends the generator $X$, but this value on $X$ is the same as before,
    \[
    (X \circ G, (X' \circ G) \otimes G') = (G, G').
    \]
    This means the $\bsD$-rig maps
    \[
    F \mapsto (F \circ G, (F \circ G)'), \qquad F \mapsto (F \circ G, (F' \circ G) \otimes G')
    \]
    coincide, and this completes the proof.
\end{proof}
\subsection*{Generalized Leibniz rule and Taylor formula}

\begin{proof}[Proof of \ref{iter_lei_spec}]\label{proof_of_iter_lei_spec}
	Expand $(F*G)^{(U)}[C] = (F*G)[C+U]$ using the fact that
	\[(F*G)[C\cup U] = \sum_{A+B=C+U}FA\times GB.\]
	For each indexing pair $(A, B)$, put $A' = A \cap C$, $B' = B \cap C$, $S = A \cap U$, $T = B \cap U$. Then $A = A' \cup S$ and $B = B' \cup T$ and $S \cup T = U$. It follows that

	\begin{align*}
		(F*G)[C+U] & \cong \sum_{A+B=C+U}FA\times GB                                             \\
		           & \cong \sum_{S \cup T = U} \sum_{A' + B' = C} F[A'\cup S] \times G[B'\cup T] \\
		           & \cong \sum_{S\cup T=U} \sum_{A'\cup B'=C} F^{(S)}[A'] \times G^{(T)}[B']    \\
		           & \cong \sum_{S \cup T = U} (F^{(S)}*G^{(T)})[C]
	\end{align*}
	This concludes the proof.
\end{proof}
\begin{proof}[Proof of \ref{taylor}]\label{proof_of_taylorella}
	Let's first observe that we have the analytic functor formula
	\[
		F(X) = \int^n F[n] \times X^n
	\]
	which mimics the Maclaurin series expansion; this is obtained from the fact that $F(\firstblank) \cong \Lan_JF$, and the integral on the right-hand side is exactly that Kan extension.

	Now given an $n$-element set $U$, let's interpret $\d^n F (A) = \d^{(U)} F (A) = F(U+A)$ as a species in the variable $n$ but as analytic in the set-variable $A$. We have then the formula
	\[
		\d^n F(A) = \int^m F[m+n] \times A^m.
	\]
	And thus we can categorify $\sum_{n=0}^\infty \frac{\d^n f(a)}{n!} x^n$ as the double coend
	\begin{multline*}
		\int^{nm} F[m+n] \times A^m \times X^n \\ \cong \int^{nmj} F[j] \times \Sigma(j, m+n) \times A^m \times X^n                       \\
		\cong \int^j F[j] \times \left(\int^{m, n} \Sigma(j, m+n) \times A^m \times X^n\right).
	\end{multline*}
	Now, we have an isomorphism
	\[
		\int^{mn} \Sigma(j, m+n) \times A^m \times X^n \cong (A + X)^j
	\]
	which ultimately comes out of the fact that $\Set$ is an extensive category: there exists an equivalence of categories $\Set/A \times \Set/X \cong \Set/(A + X)$. We conclude that
	\[
		\int^j F[j] \times (A + X)^j
	\]
	is the value $F(A+X)$ of the analytic functor $F(\firstblank)$.
\end{proof}
\begin{proof}[Proof that \eqref{a_coequ} is a coinverter]\label{proof_of_coinverting}
	The universal property of the coinverter amounts to the following:
	\begin{enumtag}{c}
		\item for each morphism of 2-rigs $p : \clC[Y] \to \clX$ such that $\initial \to p(Y^2\otimes R(Y))$ is invertible in $\clX$, there exists a unique (up to isomorphism) $\bar p : \clC[Y]_{<2} \to \clX$ such that $q\circ \bar p = p$;
		\item for each natural transformation $\alpha : p \To p'$ of 2-rig morphisms with the property that the horizontal composition $\alpha\boxminus u$ is an isomorphism, there exists a unique $\bar\alpha : \bar q \To \bar q'$ such that $q * \bar\alpha = \alpha$.
	\end{enumtag}
	Both properties descend from the fact that $p$, being a 2-rig morphism, preserves coproducts; if $p(A+BY+RY^2)\cong p(A+BY) + p(RY^2)$, and the initial arrow $\initial \to p(RY^2)$ is an isomorphism, the vertical right arrow in the commutative diagram
	\[
		\vcenter{\xymatrix{
				p(A \cup BY)  \cup  \initial \ar[r] \ar[d] & p(A \cup BY)\ar[d] \\
				p(A \cup BY)  \cup  p(RY^2) \ar[r] & p(A \cup BY \cup RY^2)
			}}
	\]
	is an isomorphism; thus, $p$ is uniquely determined by its action on $\clC[Y]_{<2}$, and $\bar p(A \cup BY)$ can be defined just as $p(A \cup BY)$.
	For what concerns 2-cells $\alpha : p \To p'$, a similar diagram
	\[
		\vcenter{\xymatrix{
		p(A \cup BY \cup RY^2) \ar[r]\ar[d]_\wr & p'(A \cup BY \cup RY^2) \ar[d]^\wr\\
		p(A \cup BY) \ar@{=}[d]\ar[r]_{\alpha_{A+BY}} & p'(A \cup BY) \ar@{=}[d]\\
		p(A \cup BY) \ar[r]_{\bar\alpha_{A+BY}} & p'(A \cup BY)
		}}
	\]
	is commutative, so $\alpha$ is uniquely determined by its components at objects $A+BY$ of $\clC[Y]_{<2}$.
\end{proof}
\end{document}

%% file: loca-macro-lics.tex
\usepackage[all,2cell]{xy}\UseAllTwocells
\usepackage{ xspace
	, hyperref
	, amsmath
	, amsthm
	, cleveref
	, booktabs
	, diagbox
	, enumitem
	, stmaryrd
	, adjustbox
	, tikz
	, xparse
	, blindtext
	, amsfonts
	, hyphenat
}

\hypersetup{ bookmarksnumbered=true
	, bookmarksopen=true
	, breaklinks=true
	, urlcolor=red!40!black
	, colorlinks%
	, citecolor=blue!40!black%
	, linkcolor=blue!40!black
}

\def\Cat{\cate{Cat}}

\def\cate#1{\mathsf{#1}}
\def\cf{see\@\xspace}

\def\d{\partial}
\def\Der{\mathsf{Der}}
\def\dopb{\begin{tikzpicture}[scale=.375] \draw (0,0) -| (1,1);\end{tikzpicture}}
\def\dopo{\begin{tikzpicture}[scale=.375] \draw (0,0) |- (1,1);\end{tikzpicture}}
\def\emdash{\text{-}}
\def\Fin{\cate{Fin}}

\def\l{\textsc{l}}
\def\lact{\mathbin{\rotatebox[origin=c]{90}{$\oslash$}}}%{\obslash}

\def\Lan{\mathrm{Lan}}

\def\Mod{\cate{Mod}}

\def\op{\text{op}}
\def\pb{\ar@{}[dr]|(.375){\dopb}}
\def\po{\ar@{}[dr]|(.675){\dopo}}

\def\r{\textsc{r}}
\def\ract{\oslash}

\def\Set{\cate{Set}}

\def\To{\Rightarrow}

\renewcommand{\cup}{\mathop{+}}

\newcommand{\leib}[2]{\d #1 \otimes #2 \cup #1 \otimes \d #2}

\theoremstyle{definition}
 \newtheorem{notation}[subsection]{Notation}
 \newtheorem{definition}[subsection]{Definition}
 \newtheorem{remark}[subsection]{Remark}
 
 \newtheorem{warning}[subsection]{Warning}
 \newtheorem{example}[subsection]{Example}

 \newtheorem{proposition}[subsection]{Proposition}
 \newtheorem{corollary}[subsection]{Corollary}
 \newtheorem{theorem}[subsection]{Theorem}

\ExplSyntaxOn
\NewDocumentCommand{\makeabbrev}{mmm}
{
	\yoruk_makeabbrev:nnn { #1 } { #2 } { #3 }
}

\cs_new_protected:Npn \yoruk_makeabbrev:nnn #1 #2 #3
{
	\clist_map_inline:nn { #3 }
	{
		\cs_new_protected:cpn { #2 } { #1 { ##1 } }
	}
}
\ExplSyntaxOff

\makeabbrev{\mathbf}{bf#1}{
	a,b,c,d,e,g,h,i,j,k,l,m,n,o,p,q,r,t,u,v,w,x,y,z,%
	A,B,C,D,E,G,H,I,J,K,L,M,N,O,P,Q,R,T,U,V,W,X,Y,Z }
\makeabbrev{\boldsymbol}{bs#1}{%
	a,b,c,d,e,f,g,h,i,j,k,l,m,n,o,p,q,r,s,t,u,v,w,x,y,z,%
	A,B,C,D,E,F,G,H,I,J,K,L,M,N,O,P,Q,R,S,T,U,V,W,X,Y,Z }
\makeabbrev{\mathsf}{sf#1}{
	a,b,c,d,e,f,g,h,i,j,k,l,m,n,o,p,q,r,s,t,u,v,w,x,y,z,%
	A,B,C,D,E,F,G,H,I,J,K,L,M,N,O,P,Q,R,S,T,U,V,W,X,Y,Z }
\makeabbrev{\mathfrak}{fk#1}{
	a,b,c,d,e,f,g,h,j,k,i,l,m,n,o,p,q,r,s,t,u,v,w,x,y,z,%
	A,B,C,D,E,F,G,H,I,J,K,L,M,N,O,P,Q,R,S,T,U,V,W,X,Y,Z }
\makeabbrev{\mathcal}{cl#1}{
	A,B,C,D,E,F,G,H,I,J,K,L,M,N,O,P,Q,R,S,T,U,V,W,X,Y,Z }
\makeabbrev{\mathbb}{bb#1}{
	A,B,C,D,E,F,G,H,I,J,K,L,M,N,O,P,Q,R,S,T,U,V,W,X,Y,Z }
\makeabbrev{\underline}{u#1}{
	a,b,c,d,e,f,g,h,i,j,k,l,m,n,o,p,q,r,s,t,u,v,w,x,y,z,
	A,B,C,D,E,F,G,H,I,J,K,L,M,N,O,P,Q,R,S,T,U,V,W,X,Y,Z }

\newlength{\seplen}
\newlength{\sepwid}
\newcommand{\firstblank}{\,\rule{\seplen}{\sepwid}\,}

\DeclareDocumentEnvironment{enumtag}{m}{%
	\begin{enumerate}%
		[ label = \textsc{#1}\oldstylenums{\arabic*}),%
			ref   = \textsc{#1}\oldstylenums{\arabic*}%
		] }{ \end{enumerate} }

\newcommand{\cop}[2]{
	\left[\begin{smallmatrix}
			#1 \\ #2
		\end{smallmatrix}\right]
}

\newenvironment{adju}[1][0.925]{%
	\begin{center}\begin{adjustbox}{max height=0.5\textheight, max width=#1\textwidth}}{\end{adjustbox}\end{center}}

\let\xto\xrightarrow

\def\[{\begin{equation}}
			\def\]{\end{equation}}

\usepackage{newtxtext,newtxmath}
\def\firstblank{-}

\def\initial{0}
\usepackage[final]{showlabels}
\usepackage{tikz-cd}

\def\Spc{\cate{Spc}}

\title{Differential 2-rigs}
\author{%
  Fosco Loregian
  \email{folore@ttu.ee}
  \institute{Tallinn University of Technology, Tallinn, Estonia}%
	\thanks{Supported by the ESF funded Estonian IT Academy research measure (project 2014--2020.4.05.19--0001).}
	\and
	Todd Trimble
	\email{topological.musings@gmail.com}
  \institute{Western Connecticut State University, Danbury, CT}}

\def\authorrunning{Loregian, Trimble}
\def\titlerunning{Differential 2-rigs}